\documentclass[twoside, 12pt]{article}
\usepackage{amssymb, amsmath, mathrsfs, amsthm}
\usepackage{graphicx}
\usepackage{color}
\usepackage[top=1in, bottom=1in, left=1in, right=1in]{geometry}
\usepackage{float, caption, subcaption}
\usepackage{amsfonts,amsmath,amssymb,amsthm}
\usepackage{graphicx}
\usepackage[linesnumbered,ruled,vlined]{algorithm2e} 
\usepackage{fixmath}
\usepackage{hyperref}
\usepackage{tikz}
\usepackage{makecell} 

\newtheorem{theo}{Theorem}[section]

\newcommand{\old}[1]{{}}

\def\emptyset{\varnothing}

\usepackage{array}
\newcolumntype{L}[1]{>{\raggedright\let\newline\\\arraybackslash\hspace{0pt}}b{#1}}
\newcolumntype{C}[1]{>{\centering\let\newline\\\arraybackslash\hspace{0pt}}b{#1}}
\newcolumntype{R}[1]{>{\raggedleft\let\newline\\\arraybackslash\hspace{0pt}}m{#1}}

\providecommand{\keywords}[1]{\textbf{\textit{Keywords---}} #1} 
\newcolumntype{M}[1]{>{\centering\arraybackslash}m{#1}}

\begin{document}

\pagestyle{plain}

\title{An Improved Algorithm for Counting Graphical Degree Sequences}

\author{
Kai Wang\footnote{Department of Computer Sciences,
Georgia Southern University,
Statesboro, GA 30460, USA 
\tt{kwang@georgiasouthern.edu}},
Troy Purvis\footnote{Department of Computer Sciences,
Georgia Southern University,
Statesboro, GA 30460, USA 
\tt{tp02863@georgiasouthern.edu}}
}
\maketitle

\begin{abstract}
We present an improved version of a previous efficient algorithm that computes the number $D(n)$ of zero-free graphical degree
sequences of length $n$. A main ingredient of the improvement
lies in a more efficient way to compute the function $P(N,k,l,s)$ of Barnes and Savage.
We further show that the algorithm can be easily adapted to compute the $D(i)$
values for all $i\le n$ in a single run. Theoretical analysis shows that the new algorithm to compute all
$D(i)$ values for $i\le n$ is a constant times faster than the previous algorithm to compute a single $D(n)$.
Experimental evaluations show that the constant of improvement is about 10.
We also perform simulations to estimate the asymptotic order of $D(n)$ by generating uniform random samples from the set of
$E(n)$ integer partitions
of fixed length $n$ with even sum and largest part less than $n$ and computing the proportion of them that are graphical degree
sequences. The known numerical results of $D(n)$ for $n\le 290$ together with the known bounds of $D(n)$ and simulation
results allow us to make an informed guess about its unknown asymptotic order. The techniques for the improved algorithm
can be applied to compute other similar functions that count the number of graphical degree
sequences of various classes of graphs of order $n$ and that all involve the function $P(N,k,l,s)$.
\end{abstract}

\section{Introduction}
We consider finite simple graphs (i.e. finite undirected graphs without loops or multiple edges) and their graphical degree sequences
that are treated as multisets (i.e. the order of the terms in the sequence does not matter).
The notion of an integer partition is well-known in number theory.
The terms in an integer partition and a graphical degree sequence are
often written in non-increasing order for convenience. An integer partition is called a graphical partition if it is the vertex
degree sequence of some simple graph. Any given integer partition $(a_1, a_2, \cdots, a_n)$ can be easily tested
whether it is a graphical partition, for example, through the Erd{\H{o}}s-Gallai criterion \cite{ErdosCallai1960}.
Clearly a zero-free graphical degree sequence and
a graphical partition are equivalent notions. The former is often used in a context where the lengths of the considered
sequences are the same and fixed and the latter is often used in a context where the sums of the parts in the considered
partitions are the same and fixed. 

It is well-known that the number of graphs of order $n$ can be efficiently calculated exactly using the Redfield-P\'{o}lya
theorem \cite{Redfield1927,Polya1937} and also asymptotically (which is $2^{n\choose 2}/n!$) based on the fact that almost
all graphs of order $n$ have only the trivial automorphism when $n$ is large \cite{HararyPalmer1973}.
Somewhat surprisingly, no algorithm was known to efficiently compute the number $D_0(n)$
of graphical degree sequences of length $n$
until recently. Previous known algorithms to compute $D_0(n)$ are from Ruskey et al. \cite{Ruskey1994}
and Iv\'{a}nyi et al. \cite{Ivanyi2013}.
Ruskey et al.'s algorithm can compute $D_0(n)$ by generating all graphical degree sequences of length
$n$ using a highly efficient ``Reverse Search'' approach which seems to run in constant amortized time.
Iv\'{a}nyi et al.'s algorithm can compute the number $D(n)$ of zero-free graphical degree sequences of length $n$ by
generating the set of all $E(n)$ integer partitions of $n$ parts with even sum and each part less than $n$ and testing
whether they are graphical partitions using linear time algorithms similar to the Erd{\H{o}}s-Gallai criterion.
The $D_0(n)$ value can be easily calculated when all $D(i)$ values for $i\le n$ are known since
$D_0(n)=1+\sum_{i=2}^nD(n)$ when $n\ge 2$ \cite{Ivanyi2013}. Burns \cite{Burns2007} proves good exponential upper
bound $O(4^n/((\log n)^c\sqrt{n}))$ and lower bound $\Omega(4^n/n)$ of $D_0(n)$ for sufficiently large $n$, although its tight
asymptotic order is still unknown. The exponential lower bound of $D_0(n)$ necessarily makes
these enumerative algorithms run in time exponential in $n$ and therefore
impractical for the purpose of computing $D_0(n)$.

In \cite{Wang2016} concise formulas and efficient polynomial time dynamic programming
algorithms have been presented to calculate $D_0(n)$, $D(n)$ and the number $D_{k\_con}(n)$ of graphical
degree sequences of $k$-connected graphs of order $n$ for every fixed $k\ge 1$ all
based on an ingenious recurrence of Barnes and Savage \cite{BarnesSavage1995}. Unfortunately the asymptotic
orders of these functions do not appear to be easily obtainable through these formulas, which is why we
currently strive to compute as many values of these functions as possible for the purpose of guessing their asymptotic trends.
Although these new algorithms for $D_0(n)$ and $D(n)$ are fast with time complexity $O(n^6)$,
they still quickly encounter bottlenecks because of the large space complexity $O(n^5)$.
This motivates us to further investigate the possibility of reducing memory usage for these algorithms.

In this paper we introduce nontrivial improvement to these algorithms, using the computation of $D(n)$ as an example,
to achieve significant memory usage reductions besides proportional run time reductions. We also show
that the algorithm can be easily adapted to compute the $D(i)$ values for all $i\le n$ in a single run with essentially
the same run time and memory usage
as computing a single $D(n)$ value. The introduced techniques can be applied to all similar algorithms
that compute the number of graphical degree sequences of various classes of simple graphs of order $n$ based on the
recurrence of Barnes and Savage. We will prove that the new algorithm that computes all $D(i)$ values for $i\le n$ achieves
a constant factor improvement in both space and time complexity than the previous algorithm in \cite{Wang2016}
that computes a single $D(n)$ value.
The experimental performance evaluations show that the constant is about 10. We also briefly mention the guessed
asymptotic order of $D(n)$ based on simulation results and the prospects of determining its unknown growth order.

\begin{table}[!htb]
	\centering
	\caption{Terminology used in this paper}
	\begin{tabular}{||c|l||}
		\hline\hline
		Term & Meaning\\
		\hline\hline
		$\mathbf{P}(N)$ & set of unrestricted partitions of an integer $N$\\
		\hline
		$\mathbf{P}(N,k,l)$ & set of partitions of an integer $N$ into at most $l$ parts\\
		& with largest part at most $k$ \\
		\hline
		$\mathbf{P}(N,k,l,s)$ & subset of $\mathbf{P}(N,k,l)$ determined by integer $s$ (See def. (\ref{eqn:PNkls}))\\
		\hline
		$\mathbf{G}^{'}(N,l)$ & set of graphical partitions of $N$ with exactly $l$ parts\\
		\hline
		$\mathbf{H}^{'}(N,l)$ & set of graphical partitions of $N$ with exactly $l$ parts\\
		& and largest part exactly $l-1$\\
		\hline
		$\mathbf{L}^{'}(N,l)$ & set of graphical partitions of $N$ with exactly $l$ parts\\
		& and largest part less than $l-1$\\
		\hline
		$\mathbf{D}(n)$ & set of zero-free graphical degree sequences of length $n$ \\
		\hline
		$\mathbf{D}_0(n)$ & set of graphical degree sequences of length $n$ allowing zero terms \\
		\hline
		$\mathbf{E}(n)$ & set of integer partitions of $n$ parts with even sum and each part $<n$ \\
		\hline
		$\mathbf{H}(n)$ & subset of $\mathbf{D}(n)$ with largest part exactly $n-1$ \\
		\hline
		$\mathbf{L}(n)$ & subset of $\mathbf{D}(n)$ with largest part less than $n-1$ \\
		\hline
		$\mathbf{I}_e(N_1,N_2)$&$\{N:N_1\le N\le N_2, N \mbox{ is an even integer}\}$\\
		\hline
		$\mathbf{I}'_e(N_1,N_2)$&$\{N:N_1\le N< N_2, N \mbox{ is an even integer}\}$\\
		\hline\hline
	\end{tabular}
	\label{tbl:definitions}
\end{table}

\section{Review of the algorithms for $D(n)$ ($L(n)$)}
\label{sec:basic_alg}
In this section we review the relevant notations, formulas and algorithms in \cite{Wang2016}.
For the reader's convenience, the terminology employed in this paper is summarized in Table \ref{tbl:definitions}.
We use bold face letters to indicate a set and the same normal face letters to indicate the cardinality of that set.
For example,
$\mathbf{P}(N,k,l)$ is the set of partitions of an integer $N$ into at most $l$ parts with largest part at most $k$ while
$P(N,k,l)$ is the cardinality of the set $\mathbf{P}(N,k,l)$, i.e. the number of partitions of an integer $N$ into at most $l$
parts with largest part at most $k$.

As shown in \cite{Wang2016}, there are concise formulas to compute $D_0(n)$, $D(n)$, $H(n)$ and $L(n)$ which all involve the
function $P(N,k,l,s)$ introduced by Barnes and Savage \cite{BarnesSavage1995}. The original definition
of the set $\mathbf{P}(N,k,l,s)$ is as follows \cite{BarnesSavage1995}:
\begin{equation} \label{eqn:PNkls}
\mathbf{P}(N,k,l,s)=\left\{ \begin{array}{ll}
\emptyset & \mbox{if $s < 0$};\\
\{\pi \in \mathbf{P}(N,k,l) : s+\sum_{i=1}^{j}r_i(\pi)\ge j, 1 \le j \le d(\pi)\} & \mbox{if $s \ge 0$}.\end{array} \right.
\end{equation}
In this definition $d(\pi)$ is the side length (number of rows) of the Durfee square of the Ferrers diagram of the integer
partition $\pi$. The function $r_i(\pi)$ is defined as $r_i(\pi)=\pi_i^{'}-\pi_i$ where
$\pi_i$ and $\pi_i^{'}$ are the number of dots in the $i$-th row and column of the Ferrers diagram of $\pi$,
respectively, for $1 \le i \le d(\pi)$. Equivalently, $\pi_i$ and $\pi_i^{'}$ are the $i$-th largest part of the
partition $\pi$ and the conjugate of $\pi$, respectively.
In the literature the values $\pi_i-\pi_i^{'}$ are called \textit{ranks} of a partition $\pi$
so the values $r_i(\pi)$ can also be called \textit{coranks} of the partition $\pi$.

The calculation of the function $P(N,k,l,s)$ need not follow the definition of $\mathbf{P}(N,k,l,s)$. Instead
it can be efficiently calculated using dynamic programming through a recurrence of Barnes and Savage \cite[Theorem 1]{BarnesSavage1995}. Our improved algorithm in the next section mainly focuses on how to compute this function
in a more efficient way.

We summarize some of the formulas from \cite{Wang2016} here:
\begin{equation} \label{eqn:D0(n)}
D_0(n)=\sum_{N\in \mathbf{I}_e(0,n(n-1))}P(N,n-1,n,0).
\end{equation}
\begin{equation} \label{eqn:D(n)}
\begin{split}
D(n)&=\sum_{N\in \mathbf{I}_e(n,n(n-1))}G^{'}(N,n) \\
&=\sum_{N\in \mathbf{I}_e(n,n(n-1))}\sum_{k=1}^{n-1}P(N-k-n+1,k-1,n-1,n-k-1).
\end{split}
\end{equation}
\begin{equation} \label{eqn:L(n)}
\begin{split}
L(n)&=\sum_{N\in \mathbf{I}_e(n, n(n-2))}L^{'}(N,n) \\
&=\sum_{N\in \mathbf{I}_e(n, n(n-2))}\sum_{k=1}^{n-2}P(N-k-n+1,k-1,n-1,n-k-1).
\end{split}
\end{equation}
These formulas can all be implemented in efficient dynamic programming algorithms that run in time polynomial in $n$
based on the recurrence of Barnes and Savage \cite[Theorem 1]{BarnesSavage1995}.

As indicated in \cite{Wang2016}, the computation of $D(n)$ can be transformed into the computation of $L(n)$ if $D_0(n-1)$
is already known based on the relation
\begin{equation} \label{eqn:D(n)L(n)}
D(n)=L(n)+D_0(n-1).
\end{equation}
The benefit of this transformation is to save memory because we only need to calculate half of the $L^{'}(N,n)$
($N\in \mathbf{I}_e(n, n(n-1)/2)$ among $N\in \mathbf{I}_e(n, n(n-2))$)
values in order to calculate $L(n)$ due to the symmetry of the sequence $L^{'}(N,n)$ in the sense that
\begin{equation} \label{eqn:L(N,n)Symmetry}
L^{'}(N,n)=L^{'}(n(n-1)-N,n), N\in \mathbf{I}_e(n, n(n-2)).
\end{equation}
This transformation makes it feasible to allocate a smaller four dimensional array, which is about one quarter
of the size for calculating $D(n)$ using formula (\ref{eqn:D(n)}) directly,
to hold the necessary $P(*,*,*,*)$ values in order to compute $L(n)$. The sequence of $G^{'}(N,n)$ values for
$N\in \mathbf{I}_e(n,n(n-1))$ used in formula (\ref{eqn:D(n)}),
though also unimodal as $L^{'}(N,n)$ for $N\in \mathbf{I}_e(n, n(n-2))$,
does not possess a similar symmetry. With this transformation in mind, we will treat the computation of $L(n)$
and $D(n)$ as equivalent problems.

\begin{algorithm}[h]
	\DontPrintSemicolon 
	\KwIn{A positive integer $n$}
	\KwOut{$L(n)$}
	$N \gets n(n-1)/2$\;
	Allocate a four dimensional array $P[N-n+1][n-2][n][N-n+1]$\;
	Fill in the array $P$ using dynamic programming based on \cite[Theorem 1]{BarnesSavage1995}\;
	$S \gets 0$\;
	\For{$i \in \mathbf{I}'_e(n,N)$ } {
		\For{$j \gets 1$ \textbf{to} $\min ({n-2,i-n+1})$ } {
			$S \gets S+P[i-j-n+1][j-1][n-1][n-j-1]$\;
		}
	}
	$S \gets 2S$\;
	\If{$N$ is even}	{ 
		\For{$j \gets 1$ \textbf{to} $\min ({n-2,N-n+1})$ } {
			$S \gets S+P[N-j-n+1][j-1][n-1][n-j-1]$\;
		}
	}
	\Return{$S$}\;
	\caption{An algorithm to compute $L(n)$.}
	\label{algo:D(n)}
\end{algorithm}

We now reiterate the pseudo-code to compute $L(n)$, and hence $D(n)$ when $D_0(n-1)$ is known, in Algorithm \ref{algo:D(n)}
from \cite{Wang2016} based on formulas (\ref{eqn:L(n)}) and (\ref{eqn:D(n)L(n)})
and the symmetry (\ref{eqn:L(N,n)Symmetry}). The variable $S$ is used to store the value of $L(n)$. Line 2 indicates
the allocation sizes for the four dimensions of the array $P$. When elements of this
array are later retrieved on line 7 and 11, we use the convention that array indices start from 0 such that
the array element $P[N][k][l][s]$ stores the function value $P(N,k,l,s)$.

As noted in \cite{Wang2016}, we can choose to allocate only size 2 for the third dimension of the
array $P$ in Algorithm \ref{algo:D(n)} since each $P(*,*,l,*)$ value depends only on the $P(*,*,l,*)$ and
$P(*,*,l-1,*)$ values according to \cite[Theorem 1]{BarnesSavage1995} and for the purpose of computing $L(n)$
only the $P(*,*,n-1,*)$ values are used on line 7 and 11.

In the next section we will introduce further improvements to this algorithm in order to save memory besides
run time.

\section{Improved algorithm for $L(n)$ ($D(n)$)}
\label{sec:improved_alg}
In this section we first show how the computation of a single $L(n)$ value can be improved in Algorithm \ref{algo:D(n)}.
A main idea is to reduce the allocation size for the fourth dimension of the array $P$ based on
some simple observations about the function $P(N,k,l,s)$ regarding its fourth variable $s$.
Then we show how the algorithm can be easily extended to compute all $L(i)$ values for $i\le n$ in a single run
with essentially the same run time and memory usage as the computation of the single $L(n)$ value.

As mentioned in Section \ref{sec:basic_alg}, we need to calculate all the $L^{'}(N,n)$ values for
$N\in \mathbf{I}_e(n, n(n-1)/2)$ in order to calculate $L(n)$, where $L^{'}(N,n)$ can be calculated as:
\begin{equation} \label{eqn:L(Nn)}
L^{'}(N,n)=\sum_{k=1}^{n-2}P(N-k-n+1,k-1,n-1,n-k-1).
\end{equation}
It is clear that the largest index of the first dimension of all the needed $P(*,*,*,*)$ values is at most
$n(n-1)/2-n=n(n-3)/2$ (corresponding to $N=n(n-1)/2$ and $k=1$). This explains why the allocation size for the first dimension
of the array $P$ in Algorithm \ref{algo:D(n)} is $n(n-3)/2+1$. In fact this allocation size can be slightly reduced.
For each pair of $N\in \mathbf{I}_e(n, n(n-1)/2)$ and $1\le k\le n-2$, each term $P(N-k-n+1,k-1,n-1,n-k-1)$ in the sum
(\ref{eqn:L(Nn)}) for $L^{'}(N,n)$ is nonzero only if
\[ N-k-n+1\le (k-1)(n-1) \]
by definition. This inequality reduces to $k\ge N/n$. This means we only need to include in the sum the $P(N-k-n+1,k-1,n-1,n-k-1)$
values for which $N-k\le N-N/n=N(1-1/n)$. Since $N\le n(n-1)/2$, the largest index of the first dimension of all the needed
non-zero $P(*,*,*,*)$ values is thus at most $\frac{n(n-1)}{2}(1-1/n)-n+1=(n^2+3)/2-2n=(n-1)(n-3)/2$, which is slightly smaller than
$n(n-3)/2$. It is also evident that the largest index of
the fourth dimension of all the needed $P(*,*,*,*)$ values for calculating each
$L^{'}(N,n)$ is $n-2$ (corresponding to $k=1$). However, this does not mean that we can simply allocate size $n-1$ for
the last dimension of the array $P$ in Algorithm \ref{algo:D(n)}.
If we examine the recurrence for $P(N,k,l,s)$ in \cite[Theorem 1]{BarnesSavage1995},
we can see that the indices of the first three dimensions never increase in any recursive computation
involving this four-variate function, while the index of the last dimension could increase because one of
the four terms $P(N,k,l,s)$ depends on is $P(N-k-l+1,k-1,l-1,s+l-k-1)$ whose index in the last dimension ($s+l-k-1$)
could be larger than $s$. The good news is that we do not need to make an allocation for the last dimension larger than
that for the first dimension since \cite[Theorem 1]{BarnesSavage1995} also ensures that $P(N,k,l,s)=P(N,k,l,N)$
for $s\ge N$. This explains why the first and fourth dimensions of the array $P$ have the same allocation sizes on line 2
in Algorithm \ref{algo:D(n)}. And based on the above discussion the allocation sizes for these two dimensions can be reduced
from $n(n-3)/2+1$ to $(n^2+5)/2-2n$.

Now we show that the allocation size $(n^2+5)/2-2n$ for the fourth dimension of the array $P$ in Algorithm \ref{algo:D(n)}
is conservative and it can be further reduced. First we recall a lemma of Barnes and Savage, on which
\cite[Theorem 1]{BarnesSavage1995} is partly based:
\begin{theo}
	\label{thm_fourth_dimension}
	\cite[Lemma 5]{BarnesSavage1995} $\mathbf{P}(N,k,l,s)=\mathbf{P}(N,k,l,N)=\mathbf{P}(N,k,l)$ for $s\ge N$.
\end{theo}
We show the condition $s\ge N$ in this theorem can be easily improved based on the original definition of
$\mathbf{P}(N,k,l,s)$, which can then be
used to further save memory space usage of Algorithm \ref{algo:D(n)}. Based on the definition in (\ref{eqn:PNkls}), if we define
an integer function $M(N,k,l)$ to be
\[ M(N,k,l)=\max_{\pi\in \mathbf{P}(N,k,l),1\le j\le d(\pi)}\{j-\sum_{i=1}^{j}r_i(\pi)\},\]
then we clearly have $\mathbf{P}(N,k,l,s)=\mathbf{P}(N,k,l,N)=\mathbf{P}(N,k,l)$ for $s\ge \max\{0,M(N,k,l)\}$. Based
on the definition of $r_i(\pi)=\pi_i^{'}-\pi_i$, the partition $\pi$ in $\mathbf{P}(N,k,l)$ that achieves the maximum in the definition
of $M(N,k,l)$ is the partition $\pi^*$ of $N$ with as many parts equal to $k$ as possible with the associated $j^*$ equal
to $d(\pi^*)$. This shows that the function $M(N,k,l)$ actually does not depend on $l$ and we can write it as $M(N,k)$.
Note that $M(N,k)$ could take negative values. For the purpose of improving Algorithm \ref{algo:D(n)}, we define
the nonnegative function
\[ M'(N,k)=\max\{0,M(N,k)\}, \]
and with this definition we clearly have $P(N,k,l,s)=P(N,k,l,N)=P(N,k,l)$ for
$s\ge M'(N,k)$. The pseudo-code for computing $M'(N,k)$ is presneted in Algorithm \ref{algo:M(N,k)} based on the
$\pi^*$ and $j^*$ mentioned above. It is easy to see that $M'(N,k)\le N$. Furthermore we observe that on average
$M'(N,k)$ is a lot smaller than $N$, which improves the condition in Theorem \ref{thm_fourth_dimension} and makes this
function a main ingredient for saving memory space of Algorithm \ref{algo:D(n)} in our improved algorithm.

\begin{algorithm}[h]
	\DontPrintSemicolon 
	\KwIn{A positive integer $N$ and a positive integer $k$}
	\KwOut{$M'(N,k)$}
	$q \gets \lfloor N/k \rfloor$\;
	$r \gets N\mod k$\;
	\If{$r=0$}	{ 
		\If{$k\ge q$} {
			\Return{$q(k-q+1)$}\;
		}
		\Else {
			\Return{$0$}\;
		}
	}
	\Else{
		\If{$k\ge q$} {
			\If{$r\le q$} {
    				\Return{$q(k-q+1)-r$}\;
    			}
    			\Else {
				\Return{$q(k-q-1)+r$}\;
    			}
		}
		\Else {
			\Return{$0$}\;
		}
	}
	\caption{Pseudo-code for computing the helper function $M'(N,k)$.}
	\label{algo:M(N,k)}
\end{algorithm}
In order to further save memory space usage of Algorithm \ref{algo:D(n)}, we define a new function
$Q(l,k,N,s)$ by reversing the order of the first three variables of the four-variate function $P(N,k,l,s)$, i.e.
\[ Q(l,k,N,s)=P(N,k,l,s). \]
With this definition a four dimensional array $Q$ can be created in the improved algorithm in place of the array
$P$ such that the allocation sizes of latter dimensions of $Q$ can be made dependent on
former dimensions and as small as possible. Specifically, the allocation size of the third
dimension of the array $Q$ can be made dependent on the first two dimensions
(explained below) and that of its fourth dimension can be made dependent on the second and third dimensions
due to the fact that $Q(l,k,N,s)=Q(l,k,N,N)=P(N,k,l)$ for $s\ge M'(N,k)$ as explained above.

Now in our improved algorithm to compute $L(n)$, the allocation size of the first
dimension of the four dimensional array $Q$ can be chosen to be 2 since, as explained before, each $Q(l,*,*,*)$
value depends only on the $Q(l,*,*,*)$ and $Q(l-1,*,*,*)$ values. The allocation size of the second dimension of the array $Q$
can be made $n-2$ since the largest index of the second dimension in all the needed $Q(*,*,*,*)$ values is $n-3$
(corresponding to $k=n-2$) based on formula (\ref{eqn:L(n)}). The allocation size of the third dimension of the array $Q$ need not
be fixed at $(n^2+5)/2-2n$ as the first dimension of $P$ in Algorithm \ref{algo:D(n)} and can be made dependent on
the indices of the first two dimensions $l$ and $k$. Specifically, it need not
exceed $lk$ since $Q(l,k,N,s)=0$ for all $N>lk$ by definition. Since we actually only allocate size 2 for the first dimension
of the array $Q$, the index $l$ cannot be used anymore and the allocation size for the third dimension of the array $Q$
can be chosen to be $\min\{k(n-1)+1,(n^2+5)/2-2n\}$ since the largest index of $l$ is $n-1$ among all the needed
$Q(l,*,*,*)$ values.
The variable allocation sizes for the third dimension effectively makes the four dimensional array $Q$ a ``ragged''
array instead of a ``rectangular'' array using common data structure terminology. The allocation size of the fourth dimension of
$Q$ can also be made variable and dependent on the indices of the second and third dimensions $k$ and $N$ respectively.
Specifically, it can be chosen to be $M'(N,k)+1$ since, as explained before, $Q(l,k,N,s)=Q(l,k,N,N)$ for $s\ge M'(N,k)$.
Many of the fourth dimensional allocation sizes $M'(N,k)+1$
are as small as 1 instead of the fixed $n(n-3)/2+1$ as in Algorithm \ref{algo:D(n)}, thereby saving a lot of memory.
We summarize the allocation sizes for the four dimensions of the array $Q$ in Table \ref{tbl:Qallocationsize}.
Since the allocation sizes for the third and fourth dimensions of the four dimensional array $Q$ in the improved
algorithm would be variable, the pseudo-code that serves the purpose of line 2 for allocation of the four
dimensional array in Algorithm \ref{algo:D(n)} now would be replaced with a revised nested loop. The improved algorithm
for $L(n)$ can be previewed in Algorithm \ref{algo:improvedD(n)}. We omit the
pseudo-code to allocate the four dimensional array $Q$ in the improved algorithm as it is not conveniently
expressible without using real programming languages.

\begin{table}[!htb]
	\centering
	\caption{Allocation sizes of the four dimensions of the array $Q$ in the improved algorithm ($l$, $k$ $N$ and $s$ are index
	variables used the nested loops in memory allocation for $Q$)}
	\begin{tabular}{||c|c||}
		\hline\hline
		Dimension (index variable) & Allocation size\\
		\hline\hline
		1st ($l$) & 2\\
		\hline
		2nd ($k$) & $n-2$\\
		\hline
		3rd ($N$) & $\min\{k(n-1)+1,\lfloor (n^2+5)/2-2n\rfloor\}$ \\
		\hline
		4th ($s$) & $M'(N,k)+1$ \\
		\hline\hline
	\end{tabular}
	\label{tbl:Qallocationsize}
\end{table}

We introduce one more improvement that would save run time of Algorithm \ref{algo:D(n)}, although not memory
space usage. In Algorithm \ref{algo:D(n)} line 3 serves to fill in the four dimensional array $P$ and it would be implemented
using nested loops. In our improved algorithm the pseudo-code to fill in the four dimensional array $Q$ would still be
implemented using nested loops with the number of iterations in the third and fourth level of the loops adjusted to accommodate
the variable allocation sizes in these two dimensions as specified in Table \ref{tbl:Qallocationsize}.
A possible improvement here is the innermost loop for the fourth dimension
of the array $Q$. We already mentioned that in the improved algorithm the allocation size for the fourth dimension
of the array $Q$ depends on the second and third dimensional indices $k$ and $N$ and is chosen to be $M'(N,k)+1$.
Instead of having an index for the fourth dimension to iterate from 0 to $M'(N,k)$ for any given index $k$ for the
second dimension and $N$ for the third dimension, we can let the index start to iterate from $m'(N,l)$ instead of
0, where $m'(N,l)=\max\{0,m(N,l)\}$ and $m(N,l)=m(N,k,l)$ is defined similarly to $M(N,k)=M(N,k,l)$ as
\[ m(N,l)=m(N,k,l)=\min_{\pi\in \mathbf{P}(N,k,l),1\le j\le d(\pi)}\{j-\sum_{i=1}^{j}r_i(\pi)\}.\]
Based on the definition of $r_i(\pi)=\pi_i^{'}-\pi_i$, the partition $\pi$ in $\mathbf{P}(N,k,l)$ that achieves the minimum
in the definition 
of $m(N,k,l)$ is the partition $\pi^\star$ of $N$ whose conjugate partition is the partition
with as many parts equal to $l$ as possible with the associated $j^\star$ equal
to $d(\pi^\star)$. This shows that the function $m(N,k,l)$ actually does not depend on $k$ and we can write it as $m(N,l)$.
Note that $m(N,l)$ could take negative values too, which is why we define the nonnegative function
$m'(N,l)=\max\{0,m(N,l)\}$ to be used as the start of the index of the innermost loop while filling in the array $Q$.
Under this definition we clearly have $Q(l,k,N,s)=0$ if $m'(N,l)>0$ and $0\le s<m'(N,l)$, which makes it feasible to
skip filling in these array elements and save time. The pseudo-code for computing $m'(N,l)$ is presented in Algorithm
\ref{algo:m(N,l)} based on the $\pi^\star$ and $j^\star$ mentioned above.

\begin{algorithm}[h]
	\DontPrintSemicolon 
	\KwIn{A positive integer $N$ and a positive integer $l$}
	\KwOut{$m'(N,l)$}
	$q \gets \lfloor N/l \rfloor$\;
	$r \gets N\mod l$\;
	\If{$r=0$}	{ 
		\If{$l\le q$} {
			\Return{$l(q-l+1)$}\;
		}
		\Else {
			\Return{$0$}\;
		}
	}
	\Else{
		\If{$l\le q$} {
			\Return{$l(q-l)+r$}\;
		}
		\Else {
			\Return{$0$}\;
		}
	}
	\caption{Pseudo-code for computing the helper function $m'(N,l)$.}
	\label{algo:m(N,l)}
\end{algorithm}

We show the pseudo-code of the improved algorithm to compute $L(n)$ in Algorithm \ref{algo:improvedD(n)}. We mainly
emphasize the part that initializes and fills in the four dimensional array $Q$ after it has been
allocated. The remaining part that computes the $L(n)$ value after the array $Q$ has been filled in is similar to Algorithm
\ref{algo:D(n)} and is abbreviated on line 10. We assume all the elements of the array $Q$ are zero after it has been
allocated. The lower bound function $m'(N,l)$ for the innermost loop variable $s$ will be needed only once for each
given pair of $N$ and $l$ while the upper bound function $M'(N,k)$ for $s$ might be needed multiple times for each
given pair of $N$ and $k$. To further save time, all the $M'(N,k)$ values can be pre-computed and later retrieved
by table lookup.

\begin{algorithm}[h]
	\DontPrintSemicolon 
	\KwIn{A positive integer $n$}
	\KwOut{$L(n)$}
	Allocate a four dimensional array $Q$ with sizes specified in Table \ref{tbl:Qallocationsize}\;
	\For{$l \gets 0$ \textbf{to} $1$ } {
		\For{$k \gets 0$ \textbf{to} $n-3$ } {
			$Q[l][k][0][0] \gets 1$\;
		}
	}
	\For{$l \gets 1$ \textbf{to} $n-1$ } {
		\For{$k \gets 1$ \textbf{to} $n-3$ } {
			\For{$N \gets 1$ \textbf{to} $\min\{lk,\lfloor (n^2+3)/2-2n\rfloor\}$ } {
				\For{$s \gets m'(N,l)$ \textbf{to} $M'(N,k)$ } {
					Update $Q[l\mod 2][k][N][s]$ using the values of $Q[l\mod 2][k-1][N][s]$, $Q[(l-1)\mod 2][k][N][s]$,
					$Q[(l-1)\mod 2][k-1][N][s]$
					and $Q[(l-1)\mod 2][k-1][N-k-l+1][s+l-k-1]$ based on \cite[Theorem 1]{BarnesSavage1995}\;
				}
			}
		}
		\tcp{Sum needed $Q[(l-1)\mod 2][*][*][*]$ values to compute $L(l)$ if desired.}
	}
	Sum $Q[(n-1)\mod 2][*][*][*]$ values to compute $L(n)$\;
	\Return{$L(n)$}\;
	\caption{Improved algorithm to compute $L(n)$ that initializes
	and fills in the four dimensional array $Q$.}
	\label{algo:improvedD(n)}
\end{algorithm}

Based on formula (\ref{eqn:L(n)}) $L(n)$ is the sum of a finite number of $P(*,*,n-1,*)$ values. After filling in
the four dimensional array $P$ in Algorithm \ref{algo:D(n)}, we can actually not only compute $L(n)$,
but also all $L(i)$ for $i=1,2,\cdots, n-1$ if we have chosen to allocate size $n$ for the third dimension instead of 2 since
all the needed $P(*,*,*,*)$ values for them are also already in the four dimensional array $P$. If we have allocated only size
2 for the third dimension, we can still compute all $L(i)$ for $i=1,2,\cdots, n$ in a single run as long as we put the loop for
the third dimension as the outermost loop when filling
in the array $P$ and compute $L(i)$ when we have already filled in all the $P(*,*,i-1,*)$ values in $P[*][*][(i-1)\mod 2][*]$
before these array elements are overwritten later.

Similarly in our improved Algorithm \ref{algo:improvedD(n)} we can also compute all $L(i)$ values for
$i=1,2,\cdots, n$ in a single run since the nested loops from line 5 to 9 ensure that all $Q(i-1,*,*,*)$ values
stored in $Q[(i-1)\mod 2][*][*][*]$ are filled in before filling in the $Q(i,*,*,*)$ values stored in
$Q[i\mod 2][*][*][*]$. We can sum all the needed $Q(i-1,*,*,*)$ values to compute $L(i)$ before they are overwritten in
the next iterations. In Algorithm \ref{algo:improvedD(n)} we add a comment at the end of the body for the outermost
loop to indicate where code can be added to compute $L(i)$ values for $i<n$ if desired. The $L(n)$ value can still be
computed on line 10 after the entire loop from line 5 to 9 has ended.

It is easy to see that the time complexity to sum all the needed $P(*,*,*,*)$ values in Algorithm \ref{algo:D(n)}
or $Q(*,*,*,*)$ values in Algorithm \ref{algo:improvedD(n)} to compute $L(n)$
is $O(n^3)$ and is of lower order than the time complexity $O(n^6)$ of filling in the array $P$
in Algorithm \ref{algo:D(n)} or the array $Q$ in Algorithm \ref{algo:improvedD(n)} (for more details see the analysis in Section
\ref{sec:analysis}). Thus computing all $L(i)$ values for $i\le n$ takes essentially the same amount of time and space
as computing the single $L(n)$ value.

\section{Complexity analysis}
\label{sec:analysis}
In this section we show that the space and time complexity of Algorithm \ref{algo:improvedD(n)} achieve an improvement
by a constant factor compared to Algorithm \ref{algo:D(n)}. We understand this is not exciting theoretically as it
is not an asymptotic improvement. However, we emphasize that the discovery of the possibility of computing all
$L(i)$ (or $D(i)$) values for $i\le n$ in a single run has its own merit which we overlooked before. Plus, the techniques in the
improved algorithm deepen our understanding of the function $P(N,k,l,s)$ and may be applied to compute other similar
functions such as $D_0(n)$ and $D_{k\_con}(n)$ and may shed insight on the analysis
of their asymptotic orders.

Now let us analyze the memory usage of Algorithm \ref{algo:D(n)} and Algorithm \ref{algo:improvedD(n)}.
Assuming an allocation size 2 for the third dimension, the total allocation size (i.e. total number of elements) for
the four dimensional array $P$ in Algorithm \ref{algo:D(n)} for computing $L(n)$ is clearly
\[ f_1(n)=2(n-2)(\frac{n(n-3)}{2}+1)^2. \] 
The allocation size $f_4(n)$ for the four dimensional array $Q$ in our improved Algorithm \ref{algo:improvedD(n)}
for computing $L(n)$ does not
appear to have a simple closed form. We now show that $f_4(n)$ and $f_1(n)$ are of the same asymptotic order, but $f_4(n)$
achieves a constant factor improvement over $f_1(n)$:
\begin{theo}
	\label{thm_space_complexity}
There exist constants $c_1$ and $c_2$ ($0<c_1<c_2<1$) such that $c_1f_1(n)\le f_4(n)\le c_2f_1(n)$ for all sufficiently large $n$.
\end{theo}
\begin{proof}
We first perform a conservative analysis of how much memory space is saved by the improved Algorithm
\ref{algo:improvedD(n)}. That is, we will obtain a lower bound of $f_1(n)-f_4(n)$.

By Table \ref{tbl:Qallocationsize} the allocation size of the
second dimension of the array $Q$ is $n-2$ so the index $k$ for the second dimension will iterate from 0 to $n-3$. For each
$0\le k\le n-3$, the allocation size of the third dimension is $\min\{k(n-1)+1,\lfloor (n^2+5)/2-2n\rfloor\}$ instead of
the fixed first dimension allocation size
$n(n-3)/2+1$ for all $0\le k\le n-3$ as in Algorithm \ref{algo:D(n)}. When $k(n-1)\le (n^2+3)/2-2n$, i.e. $k\le (n-3)/2$,
the allocation size of the third dimension of the array $Q$ is $k(n-1)+1$. Thus, due to the reduced allocation size for the
third dimension of the array $Q$, the number of saved elements from the array $Q$ compared to the array $P$
in Algorithm \ref{algo:D(n)} is at least
\[ T_1(n)=(n(n-3)/2+1)(((n-3)/2+1)(n(n-3)/2+1)-\sum_{k=0}^{(n-3)/2}(k(n-1)+1)). \]
The function $T_1(n)$ is the product of two factors. The first factor $n(n-3)/2+1$ is the allocation size of the fourth dimension of
the array $P$ in Algorithm \ref{algo:D(n)}. The second factor is the total reduction of the allocation size of the third dimension
of the array $Q$ due to variable allocations in this dimension. It is easy to see that $T_1(n)$ has an asymptotic order
of $n^5/16$.

For each index $k$ for the second dimension of the array $Q$ in the range $0\le k\le (n-3)/2$, the allocation size for the
third dimension is $k(n-1)+1$ as shown above.
Thus the index $N$ for the third dimension will iterate from 0 to $k(n-1)$. For each pair of
$k$ and $N$, the allocation size for the fourth dimension is $M'(N,k)+1$. Observe that based on Algorithm
\ref{algo:M(N,k)} we have $M'(N,k)=0$ exactly when $\lfloor N/k \rfloor>k$, i.e. when $N\ge k(k+1)$. When $N=0$ or $k=0$
we can also define $M'(N,k)=0$. Thus, due to the reduced allocation size for the
fourth dimension of the array $Q$, the number of saved elements from the array $Q$ compared to the array $P$
in Algorithm \ref{algo:D(n)} is at least
\[ T_2(n)=(n(n-3)/2)\sum_{k=0}^{(n-3)/2}(k(n-1)-k(k+1)+1). \]
The function $T_2(n)$ is the product of two factors. The first factor $n(n-3)/2$ is the amount of reduction of the allocation size
of the fourth dimension from $n(n-3)/2+1$ in Algorithm \ref{algo:D(n)} to 1 in the improved Algorithm
\ref{algo:improvedD(n)} due to $M'(N,k)=0$ when $N\ge k(k+1)$. The second factor is a sum each of whose terms
counts the number of $N$ in $[k(k+1),k(n-1)]$ since exactly these $N$ satisfy $M'(N,k)=0$. It is easy to see that
$T_2(n)$ has an asymptotic order of $n^5/24$.

The total number of saved elements $f_1(n)-f_4(n)$ from the array $Q$ compared to the array $P$
in Algorithm \ref{algo:D(n)} is at least $2(T_1(n)+T_2(n))$, which has an asymptotic order of
$5n^5/24$. There is a factor 2 in this expression because we have not included the dimension of constant
allocation size into consideration yet in the above discussion.
Since $f_1(n)$ is asymptotically $n^5/2$, we see that $f_4(n)$ is asymptotically at most $n^5/2-5n^5/24=7n^5/24$,
which means $f_4(n)\le \frac{7}{12}f_1(n)$ for all sufficiently large $n$. This analysis is conservative as some of the
$M'(N,k)$ values are zero too when $(n-3)/2<k\le n-3$ and it has not considered
the reduction of the allocation sizes of the fourth dimension where $M'(N,k)$ is nonzero but less than $n(n-3)/2+1$.
We have shown that the constant $c_2$ in the statement of the theorem can be chosen to be 7/12.

We now derive a lower bound of $f_4(n)$. For each index $k$ for the second dimension of the array $Q$ in the range
$(n-3)/2<k\le n-3$, the allocation size for the third dimension is $(n-1)(n-3)/2+1$. We have already shown that
$M'(N,k)=0$ if and only if $N\ge k(k+1)$. Thus, for $(n-3)/2<k\le n-3$ and $0\le N\le (n-1)(n-3)/2$,
if $k(k+1)>(n-1)(n-3)/2$ (i.e. $k>\frac{\sqrt{2(n-1)(n-3)+1}-1}{2}$), then $M'(N,k)\neq 0$. Clearly
$\frac{3}{4}(n-1)>\frac{\sqrt{2(n-1)(n-3)+1}-1}{2}$ when $n$ is large. Now consider the range of $k$ such that
\[ \frac{3}{4}(n-1)\le k\le n-3. \]
Each $k$ in this range can be represented as $k=c(n-1)$ for some $3/4\le c\le 1$.
Also consider the range of $N$ such that
\[ (n-1)(n-3)/4\le N\le (n-1)(n-3)/2. \]
With $k$ and $N$ in the chosen ranges,
we have $\frac{n-3}{4c}\le \frac{N}{k}\le \frac{n-3}{2c}$ and $0\le r=(N\mod k)<k=c(n-1)$ so
$\frac{n-3-4c}{4c}\le q=\lfloor \frac{N}{k} \rfloor\le \frac{n-3}{2c}$. Therefore,
$(c-\frac{1}{2c})n+\frac{3}{2c}-c\le k-q\le (c-\frac{1}{4c})n+\frac{3+4c}{4c}-c$ and
$q(k-q)\ge ((c-\frac{1}{2c})n+\frac{3}{2c}-c)\frac{n-3-4c}{4c}$. Since $c-\frac{1}{2c}\ge 1/12>0$, we have $q(k-q)=\Omega(n^2)$.
Now $q$, $q-r$, and $r-q$ are all linear in $n$, we see that $M'(N,k)=\Omega(n^2)$ for the considered ranges of
$N$ and $k$ based on Algorithm \ref{algo:M(N,k)}. The number of $k$ in the range $\frac{3}{4}(n-1)\le k\le n-3$
is $\Omega(n)$ and the number of $N$ in the range $(n-1)(n-3)/4\le N\le (n-1)(n-3)/2$ is $\Omega(n^2)$. For each
pair of $k$ and $N$ in these ranges the allocation size for the fourth dimension is $M'(N,k)+1=\Omega(n^2)$.
Consequently, the total allocation size $f_4(n)$ for the four dimensional array $Q$ in our improved Algorithm
\ref{algo:improvedD(n)} is $\Omega(n^5)$.
Since $f_1(n)$ is asymptotically $n^5/2$, we have shown that there exists a constant $c_1$ such that
$f_4(n)\ge c_1f_1(n)$ for all sufficiently large $n$ where $c_1$ can be chosen to be $1/192$. And the theorem is proved.
\end{proof}

We have collected some values of $f_4(n)$ for $n\le 1000$ and compared them with $f_1(n)$ in Table \ref{tbl:allocationComparison}.
It appears that $\frac{f_4(n)}{f_1(n)}$ is about 10\% and it is likely to tend to some constant $C$ between 0.1 and 0.2,
which agrees with the lower bound $1/192$ and upper bound $7/12$ obtained in the proof of Theorem \ref{thm_space_complexity}.
The space complexity of Algorithm \ref{algo:D(n)} and the improved Algorithm \ref{algo:improvedD(n)} are dominated
by the allocation sizes of the four dimensional array $P$ and $Q$ respectively. Thus Algorithm \ref{algo:improvedD(n)}
achieves a constant factor improvement in memory space usage.
\begin{table}[!htb]
	\centering
	\caption{Comparison of allocation sizes of Algorithm \ref{algo:D(n)} and Algorithm \ref{algo:improvedD(n)}}
	\begin{tabular}{||c|c|c|c||}
		\hline\hline
		$n$ & $f_1(n)$ & $f_4(n)$ & $f_4(n)/f_1(n)$ \\ 
		\hline\hline
		10 & 20736 & 2030 & 0.0978974 \\ 
		\hline
		20 & 1052676 & 99736 & 0.0947452 \\ 
		\hline
		30 & 9230816 & 885350 & 0.0959124 \\ 
		\hline
		40 & 41730156 & 4041722 & 0.0968537 \\ 
		\hline
		50 & 132765696 & 12948206 & 0.0975267 \\ 
		\hline
		60 & 339592436 & 33286556 & 0.0980191 \\ 
		\hline
		70 & 748505376 & 73646710 & 0.0983917 \\ 
		\hline
		80 & 1480839516 & 146132702 & 0.0986823 \\ 
		\hline
		90 & 2698969856 & 266968646 & 0.0989150 \\ 
		\hline
		100 & 4612311396 & 457104592 & 0.0991053 \\ 
		\hline
		110 & 7483319136 & 742822422 & 0.0992638 \\ 
		\hline
		120 & 11633488076 & 1156341746 & 0.0993977 \\ 
		\hline
		130 & 17449353216 & 1736425796 & 0.0995123 \\ 
		\hline
		140 & 25388489556 & 2528987092 & 0.0996116 \\ 
		\hline
		150 & 35985512096 & 3587693526 & 0.0996983 \\ 
		\hline
		300 & 1182935794196 & 118676615988 & 0.1003238 \\ 
		\hline
		400 & 5018396965596 & 504274588310 & 0.1004852 \\ 
		\hline
		500 & 15376557756996 & 1546620017330 & 0.1005830 \\
		\hline
		600 & 38364293168396 & 3861311170282 & 0.1006486 \\ 
		\hline
		700 & 83078878199796 & 8365678364352 & 0.1006956 \\ 
		\hline
		800 & 162207987851196 & 16339372522124 & 0.1007310 \\ 
		\hline
		900 & 292629697122596 & 29484953979544 & 0.1007586 \\ 
		\hline
		1000 & 496012481013996 & 49988481364570 & 0.1007807 \\ 
		\hline\hline
	\end{tabular}
	\label{tbl:allocationComparison}
\end{table}

The time complexity of the two algorithms are
dominated by the time to fill in the four dimensional array $P$ and $Q$ respectively. Although the starting index
of the variable $s$ on line 8 in Algorithm \ref{algo:improvedD(n)} is from $m'(N,l)$ instead of 0, the time complexity of
Algorithm \ref{algo:improvedD(n)} still only achieves a constant factor improvement similar to the space complexity improvement
over Algorithm \ref{algo:D(n)}. This is can be shown as follows.

Based on Algorithm \ref{algo:m(N,l)} it is easy to see that $m'(N,l)=0$ if and only if $\lfloor N/l\rfloor <l$, i.e.
$N<l^2$. When $l>\sqrt{\frac{(n-1)(n-3)}{2}}$, we have $l^2>\frac{(n-1)(n-3)}{2}$. As shown in the proof of
Theorem \ref{thm_space_complexity}, when $k$ and $N$ are in the range $\frac{3}{4}(n-1)\le k\le n-3$
and $(n-1)(n-3)/4\le N\le (n-1)(n-3)/2$, we have $M'(N,k)=\Omega(n^2)$. The number of $n$ in the range
$[\sqrt{\frac{(n-1)(n-3)}{2}},n-1]$ is $\Omega(n)$. When $l$ is in this range and $(n-1)(n-3)/4\le N\le (n-1)(n-3)/2$,
we have $m'(N,l)=0$. This shows that when $l$, $k$ and $N$ are in these specified ranges, the time to fill in
this part of the array $Q$ takes $\Omega(n^6)$ time, which becomes a lower bound of the time complexity
of Algorithm \ref{algo:improvedD(n)}. Since Algorithm \ref{algo:D(n)} takes time $O(n^6)$, we have proved
that Algorithm \ref{algo:improvedD(n)} has a time complexity of the same order as Algorithm \ref{algo:D(n)}
and achieves a constant factor improvement.

Finally we note that adding the code to compute all $L(i)$ values for $i<n$ in Algorithm \ref{algo:improvedD(n)}
does not increase the memory space usage and it will increase the run time by at most $O(n.n^3)=O(n^4)$, which is
negligible compared to the time complexity $O(n^6)$ to fill the array $Q$. Thus,
computing all $L(i)$ values for $i\le n$ has essentially the same space and time complexity as computing
a single $L(n)$ value.

\section{Experimental evaluations and simulations}
\label{sec:experiments}
We have computed the exact values of $D(n)$ for $n$ up to 290 with the help of large memory supercomputers
from XSEDE. Based on these numerical results and the known
upper and lower bound of $D_0(n)$ given in Burns \cite{Burns2007}, we had conjectured that the asymptotic order
of $D(n)$ is like $\frac{c\times 4^n}{(\log{n})^{1.5}\sqrt{n}}$ for some constant $c$. We have performed further simulations
using a method similar to that in \cite{Burns2007} to estimate the asymptotic order of $D(n)$. Based on
simulation results for $n$ up to 700000000, it seems that $D(n)$ has an asymptotic order more like
$\frac{4^n\exp(-0.7\frac{\log n}{\log\log n})}{8\sqrt{\pi n}}$. The form of this function is inspired by
Burns \cite{Burns2007} and Pittel \cite{Pittel2018}.

\section{Discussions about asymptotic orders}
We tried to derive the asymptotic order of $D(n)$ through the multi-variate generating function of the multi dimensional
sequence $P(N,k,l,s)$:
\[ F(w,x,y,z)=\sum_{N=0}^\infty \sum_{k=0}^\infty \sum_{l=0}^\infty \sum_{s=0}^\infty P(N,k,l,s)w^Nx^ky^lz^s. \]
However we are unable to obtain a simple closed-form for this generating function. The function $P(N,k,l,s)$ is quite unusual.
For one thing the last index can be increased during its recursive computation. The related function $P(N,k,l)$
actually satisfies a similar recurrence as follows:
\[ P(N,k,l)=P(N-k-l+1,k-1,l-1)+P(N,k-1,l)+P(N,k,l-1)-P(N,k-1,l-1). \]
This recurrence is simpler than that of $P(N,k,l,s)$ in \cite[Theorem 1]{BarnesSavage1995} in the sense that no index in any
of the three dimensions could increase during its recursive computation. The two share a similarity that they do not
belong to the classes of multi-variate recurrences considered by Bousquet-M\'{e}lou and Petkov\v{s}ek \cite{BOUSQUETMELOU2000}.
The single-variate generating function of the sequence $P(N,k,l)$ when $k$ and $l$ are fixed
is known from \cite{Andrews1984}, but we are unable to extend it to a multi-variate generating function.

For each given triple of $N$, $k$ and $l$ we have shown the exact range of $s$ such that $P(N,k,l,s)$ changes from minimum
to maximum. The variable $s$ can be seen to measure how close a partition in $\mathbf{P}(N,k,l,s)$ is to a graphical partition.
Fine-tuned analysis for this range of $s$ together with all the known results about the order of $P(N,k,l)$ with $k$ and $l$
in various ranges relative to $N$ might help us better understand the behavior of $P(N,k,l,s)$. The number of graphical partitions of
an even integer $N$ is shown by Barnes and Savage \cite{BarnesSavage1995} to be $G(N)=P(N,N,N,0)$.
The best results about the order of $G(N)$ we know of are from Pittel \cite{Pittel1999,Pittel2018} and the tight asymptotic
order of $G(N)$ is unknown yet. Our simulation of the asymptotic order of $G(N)$ using uniform random integer partition
generators from \cite{arratia_desalvo_2016}  suggests that $\frac{G(N)}{P(N)}$
has an asymptotic order like $e^{-\frac{0.3\log{N}}{\log\log{N}}}$, which is also inspired by the bound of $G(N)$
given in \cite{Pittel2018} and is similar to a factor of the conjectured asymptotic order of $D(n)$
given above in Section \ref{sec:experiments}.
Thus the analysis of the asymptotic behavior of $P(N,k,l,s)$ also helps to determine
the unknown asymptotic order of $G(N)$ besides the functions counting various classes of graphical degree sequences
of given length. 

\section{Conclusions}
In this paper we presented an improved algorithm to compute $D(n)$ exactly. A main ingredient of the improvement
is an analysis of the fourth dimension of the function $P(N,k,l,s)$ of Barnes and Savage
such that the exact range of $s$ in which this function varies
with given $N,k,l$ is determined and then used to help reduce memory space usage. The new algorithm makes it feasible
to compute all $D(i)$ values for $i\le n$ in about 10\% of the time that takes the previous algorithm to compute a
single $D(n)$ value. The techniques can be applied to all related
functions that can be computed exactly based on the function $P(N,k,l,s)$.

\section{Acknowledgements}
This research has been supported by a research seed grant of Georgia Southern University. The computational
experiments have been supported by the XSEDE Startup Allocation TG-DMS170020 and Research Allocation TG-DMS170025.

\bibliographystyle{plain}
\bibliography{SimDG}

\end{document}